\documentclass[12pt]{amsart}
\usepackage{amsmath,amssymb,amsfonts,amsthm,amsopn}
\usepackage{graphics}
\usepackage{mathrsfs}
\usepackage{graphicx,tikz}
\usepackage[all]{xy}
\usepackage{amsmath}
\usepackage{multirow, longtable, makecell, caption, array}
\usepackage{amssymb}
\usepackage{mathtools}
\usepackage{pb-diagram}
\usepackage{cellspace} %

\newtheorem{tm}{Theorem}[section]
\newtheorem{lemma}[tm]{Lemma}

\newtheorem{theorem}{Theorem}[section]
\newtheorem{corollary}[theorem]{Corollary}
\newtheorem{definition}[theorem]{Definition}
\newtheorem{example}[theorem]{Example}

\newtheorem{proposition}[theorem]{Proposition}
\newtheorem{remark}[theorem]{Remark}

\newcommand{\beqa}{\begin{eqnarray*}}
\newcommand{\eeqa}{\end{eqnarray*}}

\DeclareMathOperator*{\supp}{supp}

\newcommand{\field}[1]{\mathbb{#1}}
\newcommand{\bR}{\field{R}}    
\newcommand{\bN}{\field{N}}    
\newcommand{\bZ}{\field{Z}}    
\newcommand{\bC}{\field{C}}    
 \def\cF{\mathcal{F}}       
 \def\cS{\mathcal{S}}
 \def\cD{\mathcal{D}}
 
 \def\cB{\mathcal{B}}
 
 \def\cG{\mathcal{G}}
 \def\cM{\mathcal{M}}

 \def\cA{\mathcal{A}}
 
 \def\cI{\mathcal{I}}
 \def\cC{\mathcal{C}}

\def\rd{\bR^d}

\def\rdd{{\bR^{2d}}}

\def\R{\right)}

\def\<{\left<}
\def\>{\right>}

\def\mv1{M_v^1}

\newcommand{\norm}[1]{\lVert#1\rVert}

\def\R{\mathbb{R}}
\def\Ren{\mathbb{R}^d}

\def\Sn2{S_{2}(L^{2}(\Ren))}
\def\S1{S_{1}(L^{2}(\Ren))}
\def\sig00{\sigma_{0,0}}

\newcommand{\A}{\mathcal{A}}

\begin{document}

\title[Metaplectic Wigner distributions]{Metaplectic Wigner distributions}

\author{Gianluca Giacchi}
\address{Università di Bologna, Dipartimento di Matematica,  Piazza di Porta San Donato 5, 40126 Bologna, Italy; Institute of Systems Engineering, School of Engineering, HES-SO Valais-Wallis, Rue de l'Industrie 21, 1950 Sion, Switzerland; Lausanne University Hospital and University of Lausanne, Lausanne, Department of Diagnostic and Interventional Radiology, Rue du Bugnon 46, Lausanne 1011, Switzerland. The Sense Innovation and Research Center, Avenue de Provence 82
1007, Lausanne and Ch. de l’Agasse 5, 1950 Sion, Switzerland.}

\email{gianluca.giacchi2@unibo.it}


\date{}

\begin{abstract}

Metaplectic Wigner distributions were recently investigated in \cite{Bayer, CR2022, CR2021, CGR2022, Z, ZH} as natural generalizations of the classical Wigner distribution, and provide a wide class of time-frequency representations that exploits the structure of the symplectic group. This work serves as a survey on metaplectic Wigner distributions and their applications to the time-frequency analysis of modulation spaces and pseudodifferential operators, topics that are all still poorly understood. We also give some new results, generalizing Lieb's uncertainty principle to the so-called matrix Wigner distributions and proving the continuity on $M^p_{v_s}(\rd)$ spaces of metaplectic pseudodifferential operators with symbols in $M^{p'}_{v_{-s}}(\rdd)$. 

\end{abstract}

\maketitle

\section{Introduction}

Time and frequency are the main mathematical quantities to describe signals: while the importance of time evolution needs not to be justified, the knowledge of the frequency content of signals, which is encoded by the Fourier transform, is important for many practical reasons: information on the frequency source, propagation of waves through a medium and attenuation are only a few prototypical examples of physical phenomena that can be mainly explained by an insight on the frequency content of signals, cf. \cite{Cohen}. The aim of \textit{time-frequency analysis} is to describe how the frequency content of a signal evolves and to describe the time and the frequency behavior of signals simultaneously, cf. \cite{CRBook}. 

Many quantities can be used to localize signals in the time and frequency domains separately. For instance, one can measure where their energy is mostly \textit{concentrated}. However, whenever a meaningful definition of \textit{concentration} is given, uncertainty principles arise, providing mathematical obstacles to the well localization of time and frequency content of signals, cf. \cite{GrochenigBook}. Basically, shorter signals have more spread frequencies and compactly supported signals cannot be band-limited, i.e. their Fourier transform cannot be compactly supported as well, cf. \cite{PW}. \\

\textit{Time-frequency representations} come from the natural need to treat signals simultaneously in time and frequency, recovering local information that are partially lost by considering them separately. Just as a signal is decomposed into its fundamental frequencies by the Fourier transform, the \textit{short-time Fourier transform} (STFT) uses time-frequency shifts to write signals as a superposition of phase-space atoms. Basically, the idea behind its definition is simple: local frequency information can be obtained by taking the Fourier transform of a signal that is cropped on small intervals of time through a so-called \textit{window function}:
\begin{equation}\label{STFT}
	V_gf(x,\xi)=\int_{\rd}f(t)\overline{g(t-x)}e^{-2\pi i\xi\cdot t}dt, \qquad (f,g\in L^2(\rd)).
\end{equation}

The \textit{Wigner distribution} is another time-frequency representation and its definition comes from quantum mechanics, cf. \cite{Wigner, Ville}. It consists of taking the Fourier transform of the product between a past contribution and a future contribution of the signal as follows:
\[
	Wf(x,\xi)=\int_{\rd}f(x+{t}/{2})\overline{f(x-{t}/{2})}e^{-2\pi i\xi\cdot t}dt, \qquad (f\in L^2(\rd)).
\]
The theory of the STFT, of the Wigner distribution, and of many other important time-frequency representations, such as the \textit{spectrogram} and the \textit{ambiguity function}, can be found in any time-frequency analysis reference book, cf. \cite{Cohen, CRBook, FollandHAPS, GrochenigBook}.  \\

The Wigner distribution plays a central role in the theory of time-frequency representations and the reason is that it enjoys the best time-frequency resolution properties, but its non-linearity makes its analysis more complicated. In \cite{SASK}, the authors analyze the mathematical properties of some of the best known time-frequency representations. Along the years, modifications of the Wigner distribution were proposed to overcome the flaws of the Wigner distribution. Weighting differently the past and the future contributions in the definition of Wigner distribution leads to the so-called \textit{$\tau$-Wigner distributions}, cf. \cite{BDO}; or convolving the Wigner distribution with smooth kernels (\textit{Cohen's class}) may lead to more regular, or even (possibly non-negative) real-valued representations, cf. \cite{Cohen}.

Most of these representations can be related to metaplectic operators using many different approaches. In \cite{CR2022}, the authors use the metaplectic definition of the STFT and $\tau$-Wigner distributions to define the more general $\cA$-\textit{Wigner distributions}, or \textit{metaplectic Wigner distributions}, whose properties have been studied further in \cite{CR2021, CGR2022}. 
The authors suggest that metaplectic operators may turn to be the main protagonists of time-frequency analysis, revealing its true nature. They investigate the properties that metaplectic Wigner distribution shall satisfy to characterize modulation spaces, characterize which of them belong to the Cohen's class and apply the results they get to pseudodifferential operators and Shr\"odinger equations.

The so-called \textit{matrix Wigner distributions} and the related uncertainty principles are also considered in \cite{Bayer, CT2019, Z}. The symplectic matrices giving rise to matrix Wigner distributions are called \textit{totally Wigner-decomposable} in \cite{CGR2022}, where these distributions are studied in their "cross" formulation. Moreover, in \cite{ZH} the authors exploit metaplectic operators in a different way and define \textit{free metaplectic Wigner distributions} using the properties of free symplectic matrices, and study the uncertainty principles that arise in this context. \\

This work serves as a survey on metaplectic Wigner distribution theory as developed by E. Cordero, G. Giacchi, Y. He, L. Rodino and Z. Zhang in their recent works cited above and most of the results mentioned in this work have been proved in \cite{Bayer, CT2019, CR2021, CR2022, CGR2022}. In Section \ref{sec:2} we establish the preliminaries and the notation that will be used in the following; Section \ref{sec:MWD} is devoted to the definition and the first properties of metaplectic Wigner distribution as defined in \cite{CR2022}. In Section \ref{sec:CC} the Cohen's class of metaplectic Wigner distributions is characterized in terms of \textit{covariance}. In \cite{CR2021}, the authors theorize that the so-called \textit{shift-invertibility} condition is the key to characterize modulation spaces in terms of metaplectic Wigner distributions. In Section \ref{sec:MS} we report on the state-of-the art on this conjecture, as developed in \cite{CR2021, CGR2022}. Section \ref{sec:UP} is devoted to the generalization of Lieb's uncertainty principle to a particular class of metaplectic Wigner distribution, while in Section \ref{sec:MPO} we report on metaplectic pseudodifferential operators and prove that metaplectic pseudodifferential operators with symbols in certain weighted modulation spaces can be extended continuously on weighted modulation spaces. As a byproduct, we obtain a Calderon-Vaillancourt type theorem for metaplectic pseudodifferential operators.

\section{Preliminaries and notation}\label{sec:2}
\subsection{Multi-indices notation}
For a vector $\alpha\in\bN_0^d$ and $x\in\rd$, we set $x^\alpha:=x_1^{\alpha_1}\cdot\ldots\cdot x_d^{\alpha_d}$ and $|\alpha|:=\alpha_1+\ldots+\alpha_d$. If $f:\rd\to\bC$, we denote with $D^\alpha f(x):=\frac{\partial^{|\alpha|}}{\partial x_1^{\alpha_1}\ldots\partial x_d^{\alpha_d}}f(x)$.

\subsection{Test functions and distributions}
	We set
	\[
		\cS(\rd):=\Big\{f\in\cC^\infty(\rd) \ : \ \rho_{\alpha,\beta}(f):=\sup_{x\in\rd}|x^\alpha D^\beta f(x)|<\infty \ \forall \alpha,\beta\in\bN_0^d\Big\}
	\]
	the space of Schwarz functions, which is a topological vector space with the initial topology associated to the family of seminorms $\{\rho_{\alpha,\beta}\}_{\alpha,\beta\in\bN_0}$. The topological dual of $\cS(\rd)$ is called the space of tempered distribution and it is equipped with the weak-$\ast$ topology. For all $0<p\leq\infty$, $ \cS(\rd)\hookrightarrow L^p(\rd)\hookrightarrow  \cS'(\rd)$ and the inclusions are dense if $p\neq\infty$.
	
	We consider the sesquilinear duality pairing between $\cS'(\rd)$ and $\cS(\rd)$, that is the unique extension to $\cS'(\rd)\times\cS(\rd)$ of the sesquilinear inner product of $L^2(\rd)$. Namely, for all $(f,\varphi)\in\cS'(\rd)\times \cS(\rd)$,
	\[
		\langle f,\varphi\rangle:=f(\bar\varphi).
	\] 
	The reader may refer to \cite{Folland} for the main definitions and properties of operations on $\cS'$, such as differentiation and convolution. We will denote with $\langle\cdot,\cdot\rangle_{L^2(\bR^k)}$ the sesquilinear inner product on $L^2(\bR^k)$ and write simply $\langle\cdot,\cdot\rangle$ instead of $\langle\cdot,\cdot\rangle_{L^2(\rd)}$.
	
	
	\subsection{Tensor products}
		For $f,g:\rd\to\bC$, we set $f\otimes g(x,y):=f(x)g(y)$, the \textit{tensor product} of $f$ and $g$.
		The operation $\otimes$ is continuous on $L^2(\rd)$ and $\cS(\rd)$. The space $$ L(\cS)= \Big\{\sum_{j=1}^Nc_jf_j\otimes g_j, \ N\in\bN, \ c_j\in\bC, \ f_j,g_j\in\cS(\rd)\Big\} $$ is dense in $\cS(\rdd)$. If $f,g\in\cS'(\rd)$, $f\otimes g$ is defined as the unique tempered distribution characterized by its action on $L(\cS)$ by:
	\[
		\langle f\otimes g,\sum_{j=1}^Nc_j\varphi_j\otimes\psi_j\rangle=\sum_{j=1}^N\overline{c_j}\langle f,\varphi_j\rangle\langle g,\psi_j\rangle.
	\]
	The operation $\otimes:\cS'(\rd)\times\cS'(\rd)\to\cS'(\rdd)$ is continuous and the space $$ L(\cS') := \Big\{\sum_{j=1}^Nc_jf_j\otimes g_j, \ N\in\bN, \ c_j\in\bC, \ f_j,g_j\in\cS'(\rd)\Big\} $$ is dense in $\cS'(\rdd)$. Moreover $L(\cS)$ (as well as $L(L^p)$, defined analogously) is also dense in $L^p(\rdd)$ for all $1\leq p<\infty$. 
	
	
\subsection{Fourier transform}\label{subsec:FT}
	For $f\in\cS(\rd)$, we define the Fourier transform of $f$ as the function
	\begin{equation}\label{FT}
		\hat f(\xi):=\int_{\rd}f(x)e^{-2\pi i\xi\cdot x}dx.
	\end{equation}
	The Fourier transform operator $\cF f:=\hat f$ is a surjective isomorphism of $\cS(\rd)$ to itself and extends to a surjective isomorphism of $\cS'(\rd)$ to itself by duality as
	\[
		\langle \hat f,\varphi \rangle = \langle f,\cF^{-1}\varphi \rangle.
	\]
	Moreover, for all $1\leq p\leq 2$, $\cF:L^p(\rd)\to L^{p'}(\rd)$ and if $p=2$, $\cF$ is a unitary operator. If $f\in \cS(\rdd)$, we define the partial Fourier transform $\cF_2$ of $f$ as
	\[
		\cF_2 f(x,\eta):=\int_{\rd}f(x,y)e^{-2\pi i\eta\cdot y}dy.
	\]
	Clearly, if $f,g\in \cS(\rd)$, $\cF_2(f\otimes g)=f\otimes\hat g$. The partial Fourier transform $\cF_2$ is also defined on $\cS'(\rdd)$ as
	\[
		\langle\cF_2 f,\varphi \rangle:=\langle f,\cF_2^{-1}\varphi\rangle, \qquad \varphi\in\cS(\rdd).
	\] 

\subsection{Time-frequency shifts}
	If $x,\xi\in\rd$ and $f:\rd\to\bC$ is measurable, we set
	\[
		T_xf(t):=f(t-x) \qquad \text{and} \qquad M_\xi f(t):=e^{2\pi i\xi\cdot t}f(t).
	\]
	$T_x$ is called \textit{translation operator} and $M_\xi$ is called \textit{modulation operator}. It is immediate to verify that for all $x,\xi\in\rd$, $T_x$ and $M_\xi$ are unitary operators on $L^2(\rd)$. Moreover, if $f\in\cS'(\rd)$, $T_xf$ and $M_\xi f$ are the tempered distribution characterized by their action on $\varphi\in\cS(\rd)$ as
	\[
		\langle T_xf,\varphi\rangle:= \langle f,T_{-x}\varphi\rangle  \qquad \text{and} \qquad \langle M_\xi f,\varphi\rangle:=\langle f,M_{-\xi}\varphi\rangle.
	\]
	Finally, if $f\in\cS'(\rd)$, $\cF(T_xf)=M_{-x}\hat f$ and $\cF(M_\xi f)=T_{\xi}f$. For this reason, the operators $\pi(x,\xi):=M_\xi T_x$ are called \textit{time-frequency shifts}.
	\subsection{Mixed norm spaces}
	For $s\in\bR$, we set $v_s(z):=(1+|z|^2)^{s/2}$ ($z\in\rd$). We say that a non-negative function $m$ on $\rd$ is \textbf{$v_s$-moderate} if $m(z+w)\lesssim v_s(z)m(w)$ for all $z,w\in\rd$.
	
	For $0<p,q\leq\infty$ and $f:\rdd\to\bC$ measurable, we set 
	\[
		\Vert f\Vert_{L^{p,q}}:=\left(\int_{\rd}\left(\int_{\rd}|f(x,y)|^pdx\right)^{q/p}dy\right)^{1/q},
	\]
	with the obvious modification for the case in which $\max\{p,q\}=\infty$. If $m$ is $v_s$-moderate for some $s\in\bR$, we set $L^{p,q}_m(\rdd)$ as the space of all measurable functions $f$ on $\rdd$ such that $\Vert mf\Vert_{L^{p,q}}<\infty$.

	\subsection{Modulation spaces}\label{subsec:MS}
	Let $g\in\cS(\rd)\setminus\{0\}$. We define the \textbf{short-time Fourier transform} (STFT) of a tempered distribution $f\in\cS'(\rd)$ with respect to the window $g$ as the function (and tempered distribution)
	\[
		V_gf(x,\xi):=\langle f,M_\xi T_xg\rangle.
	\]
	It follows directly by the definition that if $f,g\in L^2(\rd)$ and $x,\xi\in\rd$, then $|V_gf(x,\xi)|\leq\norm{f}_2\norm{g}_2$.
	
	For $0<p,q\leq\infty$ and $m:\rdd\to\bC$ $v_s$-moderate, the \textbf{modulation space} $M^{p,q}_m$ is defined as
	\[
		M^{p,q}_m(\rd):=\Big\{f\in\cS'(\rd):V_gf\in L^{p,q}_m(\rdd)\Big\}.
	\]
	If $p=q$, we simplify the notation writing $M^p_m$ instead of $M^{p,p}_m$. These are Banach spaces (quasi-Banach if $0<\min\{p,q\}<1$) with respect to the (quasi) norm $\norm{f}_{M^{p,q}_m}=\norm{V_gf}_{L^{p,q}_m}$ and different windows $g\in\cS(\rd)\setminus\{0\}$ define equivalent (quasi) norms. It can be proved that $\cS(\rd)\subseteq M^{p,q}_m(\rd)$ and we set $\cM^{p,q}_m(\rd)$ to denote the closure of $\cS(\rd)$ in $M^{p,q}_m(\rd)$, which coincides with the latest if $p,q\neq\infty$. Moreover, $M^{p_1,q_1}_m(\rd)\hookrightarrow M^{p_2,q_2}_m(\rd)$ if $0<p_1\leq p_2\leq\infty$ and $0<q_1\leq q_2\leq\infty$ and $M^2(\rd)=L^2(\rd)$. Moreover, $(M^{p,q}_m(\rd))'=M^{p',q'}_{1/m}(\rd)$ for $1\leq p,q<\infty$. We refer to \cite{CRBook} for the other properties of modulation spaces, such as convolution relations.
	
	We recall other inclusion properties, in terms of weights and $L^{p,q}$ spaces, that will be used in this work. Namely, if two moderate weights $m_1,m_2$ satisfy $m_1\lesssim m_2$ and $0<p,q\leq\infty$, then $M^{p,q}_{m_2}(\rd)\hookrightarrow M^{p,q}_{m_1}(\rd)$. Here, the notation $m_1\lesssim m_2$ means that there exists $C>0$ such that $m_1(x)\leq Cm_2(x)$ for all $x$. Finally, if $0<p\leq 2$, then $M^p_{v_s\otimes 1}(\rdd)\hookrightarrow L^p_{v_s}(\rdd)$. 
	
	\subsection{Symplectic group}
	In this work, we set 
	\begin{equation}\label{defJ}
		J_d=\begin{pmatrix}
			0_{d\times d} & I_{d\times d}\\
			-I_{d\times d} & 0_{d\times d}
		\end{pmatrix},
	\end{equation}
	where $I_{d\times d}\in\bR^{d\times d}$ is the identity matrix and $0_{d\times d}$ is the matrix of $\bR^{d\times d}$ having all zero entries. When $d$ is clear from the context, we omit all the subscripts.
	
	A matrix $\cA\in\bR^{2d\times 2d}$ is symplectic, and we write $\cA\in Sp(d,\bR)$, if $\cA^TJ\cA=J$. Writing $\cA$ as a block matrix, i.e. 
	\begin{equation}\label{blocksA}
		\cA=\begin{pmatrix} A & B\\
		C & D\end{pmatrix},
	\end{equation}
	($A,B,C,D\in\bR^{d\times d}$) it is immediate to verify that $\cA\in\bR^{2d\times 2d}$ is symplectic if and only if the following conditions hold:
	\begin{equation}\label{carSyymp}
		\begin{cases}
			AC^T=A^TC,\\
			BD^T=B^TD,\\
			A^TD-C^TB=I.
		\end{cases}
	\end{equation}
	If $\cA\in Sp(d,\bR)$, then $\det(\cA)=1$ and the inverse of $\cA$ is explicitly given in terms of the blocks of $\cA$ as
	\[
		\cA^{-1}=\begin{pmatrix} D^T & -B^T\\
		-C^T & A^T\end{pmatrix}.
	\]
	If $\cA\in Sp(d,\bR)$ and $\det(B)\neq0$, $\cA$ is called \textbf{free}. For $L\in GL(d,\bR)$ and $C\in\bR^{d\times d}$, $C$ symmetric, we define
	\begin{equation}\label{defDLVC}
		\cD_L:=\begin{pmatrix}
			L^{-1} & 0\\
			0 & L^T
		\end{pmatrix} \qquad \text{and} \qquad V_C:=\begin{pmatrix}
			I & 0\\ C & I
		\end{pmatrix}.
	\end{equation}
	$J$ and the matrices in the form $V_C$ ($C$ symmetric) and $\cD_L$ ($L$ invertible) generate the group $Sp(d,\bR)$. Also, if $\cA\in Sp(d,\bR)$, there exist $\cA_1,\cA_2\in Sp(d,\bR)$ free such that $\cA=\cA_1\cA_2$. 

\subsection{Metaplectic operators}
Let $\rho$ be the Schr\"odinger representation of the Heisenberg group, that is $$\rho(x,\xi,\tau)=e^{2\pi i\tau}e^{\pi i\xi\cdot x}\pi(x,\xi),$$ for all $x,\xi\in\rd$, $\tau\in\bR$. For all $\cA\in Sp(d,\bR)$, $\rho_\cA(x,\xi,\tau):=\rho(\cA (x,\xi),\tau)$ defines another representation of the Heisenberg group that is equivalent to $\rho$, i.e. there exists a unitary operator $\mu(\cA):L^2(\rdd)\to L^2(\rdd)$ such that
\begin{equation}\label{muAdef}
	\mu(\cA)\rho(x,\xi,\tau)\mu(\cA)^{-1}=\rho(\cA(x,\xi),\tau) \qquad (\forall x,\xi\in\rd, \ \tau\in\bR).
\end{equation}
This operator is not unique, but if $\mu'(\cA)$ is another unitary operator satisfying (\ref{muAdef}), then $\mu'(\cA)=c\mu(\cA)$, for some unitary constant $c\in\bC$, $|c|=1$. Any of the operators $\mu(\cA)$ is called \textbf{metaplectic operator} associated to $\cA$. 

The choice of $c$ is relevant when composing metaplectic operators: if $\cA,\cB\in Sp(d,\bR)$, then $\mu(\cA)\mu(\cB)=c\mu(\cA\cB)$ for some $c\in\bC$ with $|c|=1$ and, in general, $c\neq1$. In this work, we avoid the technicality of choosing the correct constant and just write \textit{up to a unitary phase factor} whenever its presence is left implicit. For instance, the relation (\ref{muAdef}) is equivalent to
\begin{equation}\label{muAdefP}
	\mu(\cA)\pi(x,\xi)\mu(\cA)^{-1}=\pi(\cA(x,\xi)) \qquad (\forall x,\xi\in\rd)
\end{equation}
up to a unitary factor. 

\begin{proposition}{\cite[Proposition 4.27]{FollandHAPS}}\label{Folland427}
	For all $\cA\in Sp(d,\bR)$, the operator $\mu(\cA)$ maps $\cS(\rd)$ isomorphically to $\cS(\rd)$ and it extends to an isomorphism on $\cS'(\rd)$.
\end{proposition}

\begin{example}\label{es22} For particular choices of $\cA\in Sp(d,\bR)$, $\mu(\cA)$ is known. Let $J$, $\cD_L$ and $V_C$ be defined as in (\ref{defJ}) and (\ref{defDLVC}) respectively. Moreover, define for $C\in\R^{d\times d}$,
\begin{equation}\label{chirp}
	\Phi_C(t)=e^{i\pi t\cdot Ct}.
\end{equation}
Recall that if $C$ is also symmetric and invertible, then $\widehat{\Phi_C}=|\det(C)|\Phi_{-C^{-1}}$. Then, up to a unitary factor,
\begin{itemize}	
\item[{(1)}] $\mu(J)f=\cF f$;
\item[{(2)}] $\mu(\cD_L)f=|\det(L)|^{1/2} f(L\cdot)$;
\item[{(3)}] $\mu(V_C)f=\Phi_C\cdot f$ and $\mu(V_C^T)f=\widehat{\Phi_{-C}}\ast f$. 
\item[{(4)}] If $\cA\in Sp(d,\bR)$ is free with block decomposition (\ref{blocksA}), then
\[
	\mu(\cA)f(x)=(\det(B))^{-1/2}\Phi_{-DB^{-1}}(x)\int_{\rd}f(y)e^{2\pi i B^{-1}x\cdot y}\Phi_{-B^{-1}A}(y)dy.
\]
\item[{(5)}] If $\cA_{FT2}\in Sp(2d,\bR)$ is the $4d\times4d$ matrix with block decomposition
\[
	\cA_{FT2}:=\begin{pmatrix}
	I_{d\times d} & 0_{d\times d} & 0_{d\times d} & 0_{d\times d}\\
	0_{d\times d} & 0_{d\times d} & 0_{d\times d} & I_{d\times d} \\
	0_{d\times d} & 0_{d\times d} & I_{d\times d} & 0_{d\times d}\\
	0_{d\times d} & -I_{d\times d} & 0_{d\times d} & 0_{d\times d}
	\end{pmatrix},
\]
then $\mu(\cA_{FT2})=\cF_2$.
\end{itemize}
\end{example}

%

We conclude this paragraph recalling a continuity result on modulation spaces of metaplectic operators. We refer to \cite[Theorem 2.13]{CR2021} for its proof.

\begin{theorem}\label{thm213}
	Assume $s\in\bR$, $\cA\in Sp(d,\bR)$ and $0<p<\infty$. Then, $\mu(\cA):\cS(\rd)\to\cS'(\rd)$ extends to a continuous operator on $M^p_{v_s}(\rd)$ and on $\cM^\infty_{v_s}(\rd)$.
\end{theorem}

\section{Gabor Frames}
Given a discrete subset $\Lambda\subseteq\rdd$ and a window function $g\in L^2(\rd)$, the  \textbf{Gabor system} $\cG(g,\Lambda)$ is the family
\[
	\cG(g,\Lambda):=\{\pi(\lambda)g, \ \lambda\in\Lambda\}.
\]
A Gabor system $\cG(g,\Lambda)$ defines a \textbf{Gabor frame} if there exist $A,B>0$ such that
\[
	A\norm{f}_2^2\leq\sum_{\lambda\in\Lambda}|\langle f,\pi(\lambda)g\rangle|^2\leq B\norm{f}_2^2
\]
holds for all $f\in L^2(\rd)$. It is common practice to consider lattices $\Lambda=A\bZ^{2d}$, for $A\in GL(2d,\bR)$.

\section{Metaplectic Wigner distributions}\label{sec:MWD}
\subsection{Definition and first properties} Many of the most used time-frequency representations are related to metaplectic operators. 
\begin{example}\label{exDist} In what follows, $f,g\in L^2(\rd)$.
	\begin{itemize}
		\item[{(1)}] The STFT of $f$ with respect to the window $g$, defined as in (\ref{STFT}),  satisfies 
		\begin{equation}\label{metaSTFT}
			V_gf=\mu(A_{ST})(f\otimes \bar g),
		\end{equation}
		where $A_{ST}\in Sp(2d,\bR)$ is the symplectic matrix whose $d\times d$ blocks are given by
		\[
			A_{ST}=\begin{pmatrix}
				I & -I & 0 & 0\\
				0 & 0 &I & I\\
				0 & 0 & 0 & -I\\
				-I & 0 & 0 & 0
			\end{pmatrix}.
		\]
		\item[{(2)}] For $\tau\in[0,1]$, the (cross) $\tau$-Wigner distribution,
		\[
			W_\tau (f,g)(x,\xi)=\int_{\rd}f(x+\tau t)\overline{g(x-(1-\tau)t)}e^{-2\pi i\xi t}dt,
		\]
		satisfies
		\begin{equation}\label{metaWtau}
			W_\tau(f,g)=\mu(A_\tau)(f\otimes \bar g),
		\end{equation}
		where
		\begin{equation}\label{Atau}
			A_\tau=\begin{pmatrix}
				(1-\tau)I & \tau I & 0 & 0\\
				0 & 0 & \tau I & -(1-\tau)I\\
				0 & 0 & I & I\\
				-I & I & 0 & 0
			\end{pmatrix}.
		\end{equation}
		This case covers the classical (cross)-Wigner distribution ($\tau=1/2$):
		\begin{equation}\label{WignerClassica}
			W(f,g)(x,\xi)=\int_{\rd}f\left(x+\frac{t}{2}\right)\overline{g\left(x-\frac{t}{2}\right)}e^{-2\pi it\cdot\xi}dt,
		\end{equation}
		and the (cross)-Rihacek distribution $$W_0(f,g)(x,\xi)=f(x)\overline{\hat g(\xi)}e^{-2\pi i\xi\cdot x},$$ that are used as time-frequency representations in the definition of Weyl and Kohn-Nirenberg quantizations for pseudodifferential operators, respectively. The Wigner distribution is the cross-Wigner distribution when $f=g$, and analogously for the other representations. However, as when it does not cause confusion, we drop the term "cross" also when $f\neq g$.
	\end{itemize}
\end{example}
Metaplectic time-frequency representations are defined as generalizations of (\ref{metaSTFT}) and (\ref{metaWtau}).

\begin{definition}
	Let $\cA\in Sp(2d,\bR)$. The (cross) \textbf{metaplectic Wigner distribution} (or $\cA$-Wigner distribution) with symplectic matrix $\cA$ is defined as
	\begin{equation}
		W_\cA(f,g)(x,\xi):=\mu(\cA)(f\otimes \bar g)(x,\xi) \qquad (x,\xi\in \rd, \ f,g\in L^2(\rd)).
	\end{equation}
\end{definition}

By the continuity of the tensor product and Proposition \ref{Folland427}, we have an immediate boundedness result for these time-frequency representations on $\cS$, $L^2$ and $\cS'$.

\begin{proposition}
	Let $\cA\in Sp(d,\bR)$. \\
	(i) $W_\cA:L^2(\rd)\times L^2(\rd)\to L^2(\rdd)$ is continuous;\\
	(ii) $W_\cA:\cS(\rd)\times\cS(\rd)\to \cS(\rdd)$ is continuous;\\
	(iii) $W_\cA:\cS'(\rd)\times\cS'(\rd)\to\cS'(\rdd)$ is continuous. 
\end{proposition}

In what follows, we will decompose symplectic $4d\times4d$ matrices into blocks in different ways. Generally, if $\cA\in Sp(2d,\bR)$, we can write its block decomposition into the blocks $A_{ij}\in\bR^{d\times d}$ ($i,j=1,\ldots,4$) as
\begin{equation}\label{blockDA}	
\cA=\begin{pmatrix}
	A_{11} & A_{12} & A_{13} & A_{14}\\
	A_{21} & A_{22} & A_{23} & A_{24}\\
	A_{31} & A_{32} & A_{33} & A_{34}\\
	A_{41} & A_{42} & A_{43} & A_{44}
\end{pmatrix}.
\end{equation}

The first properties of metaplectic Wigner distributions were studied in \cite{CR2021,CR2022, CGR2022}. We begin with the proving that the fundamental identity of time-frequency analysis, that holds for the STFT, holds for all metaplectic Wigner distributions as well, under the following form, cf. \cite[Proposition 2.7]{CR2022}.

\begin{proposition}
	Let $\cA\in Sp(2d,\bR)$ have block decomposition (\ref{blockDA}). Then, for all $f,g\in L^2(\rd)$, up to a phase factor,
	\[
		W_{\cA}(\hat f,\hat g)=W_{\cA'}(f,g),
	\]
	where $\cA'$ is the matrix with block decomposition
	\begin{equation}
		\cA'=\begin{pmatrix}
	-A_{13} & A_{14} & A_{11} & -A_{12}\\
	-A_{23} & A_{24} & A_{21} & -A_{22}\\
	-A_{33} & A_{34} & A_{31} & -A_{32}\\
	-A_{43} & A_{44} & A_{41} & -A_{42}
\end{pmatrix}.
	\end{equation}
\end{proposition} 
\begin{proof}
	Let $S=\begin{pmatrix} I & 0\\ 0 & -I
	\end{pmatrix}$. Then, up to a unitary factor,
	\[\begin{split}
		W_\cA(\hat f,\hat g)&=\mu(\cA)(\hat f\otimes \overline{\hat g})=\mu(\cA)(\hat f\otimes \mathcal{I}\hat{\bar g})=\mu(\cA)\mu(\cD_S)\mu(J)(f\otimes \bar g)\\
		&=\mu(\cA\cD_S J)(f\otimes \bar g)=W_{\cA\cD_S J}(f,g),
	\end{split}\]
	where $\cI \varphi(t):=\varphi(-t)$ is the reflection operator. It remains to observe that $\cA\cD_S J=\A'$, since $S^{-1}=S^T=S$.
\end{proof}

Also, Moyal's identity extends to metaplectic Wigner distributions as a direct consequence of the unitarity of $\mu(\cA)$, cf. \cite[Proposition 2.9]{CR2022}:

\begin{proposition}
	Let $\cA\in Sp(2d,\bR)$. Then, for all $f_1,f_2,g_1,g_2\in L^2(\rd)$,
	\begin{equation}\label{moyal}
		\langle W_\cA(f_1,g_1),W_\cA(f_2,g_2)\rangle_{L^2(\rdd)}=\langle f_1,f_2\rangle_{L^2(\rd)}\overline{\langle g_1,g_2\rangle_{L^2(\rd)}}.
	\end{equation}
	In particular, for all $f,g\in L^2(\rd)$,
	\[
		\Vert W_\cA(f,g)\Vert_{L^2(\rdd)}=\Vert f\Vert_{L^2(\rd)}\Vert g\Vert_{L^2(\rd)}.
	\]
\end{proposition}

We conclude this section by proving a relation between the STFT and the metaplectic Wigner distributions in term of the inner product of $L^2$ and a continuity result on modulation spaces.

\begin{proposition}\label{prop36}
	Let $f,g_1,g_2,g_3\in L^2(\rd)$, $\cA\in Sp(2d,\bR)$ and assume that $\langle g_1,g_2\rangle\neq0$. Then, for all $w\in\rdd$,
	\[
		V_{g_3}f(w)=\frac{1}{\langle g_2,g_1\rangle}\langle W_\cA(f,g_1),W_\cA(\pi(w)g_3,g_2)\rangle_{L^2(\rdd)}.
	\]
\end{proposition}
\begin{proof}
	Let $f,g_1,g_2,g_3\in L^2(\rd)$ satisfy the hypothesis. Using (\ref{moyal}),
	\begin{align*}
		V_{g_3}f(w)\overline{\langle g_1,g_2\rangle}&=\langle f,\pi(w)g_3\rangle\overline{\langle g_1,g_2\rangle}=\langle W_\cA(f,g_1),W_\cA(\pi(w)g_3,g_2)\rangle,
	\end{align*}
	and the assertion follows.
\end{proof}
%

\begin{theorem}\label{thm37}
	Assume $f,g\in M^p_{v_s}(\rd)$, $0<p\leq\infty$ and $s\geq0$. For any $\cA\in Sp(2d,\bR)$, $W_\cA(f,g)\in M^p_{v_s}(\rdd)$ with
	\begin{equation}\label{conMpWA}
		\norm{W_\cA(f,g)}_{M^p_{v_s}}\lesssim\norm{f}_{M^p}\norm{g}_{M^p_{v_s}}+\norm{g}_{M^p}\norm{f}_{M^p_{v_s}}.
	\end{equation}
\end{theorem}
\begin{proof}
	By Theorem \ref{thm213}, $\mu(\cA):M^p_{v_s}(\rdd)\to M^p_{v_s}(\rdd)$ is continuous. Hence,
	\[
		\norm{W_\cA(f,g)}_{M^p_{v_s}}=\norm{\mu(\cA)(f\otimes\bar g)}_{M^p_{v_s}}\lesssim\norm{f\otimes\bar g}_{M^p_{v_s}}.
	\]
	Since $v_s(x,y)\asymp (v_s\otimes 1)(x,y)+(1\otimes v_s)(x,y)$, it follows that
	\begin{align*}
		\norm{W_\cA(f,g)}_{M^p_{v_s}}&\lesssim\norm{f\otimes\bar g}_{M^p_{v_s\otimes1+1\otimes v_s}}\lesssim \norm{f\otimes\bar g}_{M^p_{v_s\otimes 1}}+\norm{f\otimes\bar g}_{M^p_{1\otimes v_s}}\\
		&=\norm{f}_{M^p_{v_s}}\norm{g}_{M^p}+\norm{f}_{M^p}\norm{g}_{M^p_{v_s}}.
	\end{align*}
\end{proof}


\section{Cohen's class of metaplectic Wigner distributions}\label{sec:CC}
A time-frequency representation $Q=Q(f,g)$ belongs to the Cohen's class if it can be written as a convolution of the classical Wigner distribution $W$ with a so-called Cohen's kernel $\sigma\in\cS'(\rd)$. Namely,
\begin{equation}\label{cohen}
	Q(f,g)(x,\xi)=(W(f,g)\ast \sigma)(x,\xi) \qquad f,g\in L^2(\rd), \ x,\xi\in\rd.
\end{equation}
We may refer to these representations as to as \textbf{Cohen's distributions}.

Time-frequency representations with different degrees of regularity can be produced properly choosing the kernel $\sigma$. In this section we characterize the metaplectic time-frequency representations that satisfy (\ref{cohen}) and their respective kernels. Namely, we show that $\cA\in Sp(2d,\bR)$ defines a metaplectic Wigner distribution that belongs to the Cohen's class if and only if $W_\cA$ is covariant, according to the following definition.

\begin{definition}
	Let $\cA\in Sp(2d,\bR)$. We say that $W_\cA$ is \textbf{covariant} if 
	\[
		W_\cA(f,g)(\pi(z)f,\pi(z)g)=T_zW_\cA(f,g)
	\]
	holds for all $f,g\in L^2(\rd)$ and all $z\in\rdd$.
\end{definition}

We have a complete characterization of symplectic matrices that give rise to covariant metaplectic Wigner distributions. We refer to \cite[Proposition 4.4]{CR2022} and \cite[Proposition 2.10]{CR2021}.

\begin{theorem}\label{charcovidddi}
	Let $\cA\in Sp(2d,\bR)$. $W_\cA$ is covariant if and only if the block decomposition (\ref{blockDA}) is in the form
	\[
		\cA=\begin{pmatrix}
		A_{11} & I-A_{11} & A_{13} & A_{13}\\
		A_{21} & -A_{21} & I-A_{11}^T & -A_{11}^T\\
		0 & 0 & I & I\\
		-I & I & 0 & 0
		\end{pmatrix},
	\]
	with $A_{13}=A_{13}^T$ and $A_{21}=A_{21}^T$.
\end{theorem}
\begin{proof}
	Up to a phase factor, using the fact that $\pi(z)f\otimes\overline{\pi(z)g}=\pi(z_1,z_1,z_2,-z_2)(f\otimes\bar g))$ and (\ref{muAdefP}), we have
	\begin{align*}
		W_\cA(\pi(z)f,\pi(z)g)&=\mu(\cA)(\pi(z)f\otimes\overline{\pi(z)g})=\mu(\cA)(\pi(z_1,z_1,z_2,-z_2)(f\otimes\bar g))\\
		&=\pi(\cA(z_1,z_1,z_2,-z_2))W_\cA(f,g)
	\end{align*}
	for all $z=(z_{1},z_2)\in \rdd$ and all $f,g\in L^2(\rd)$. Since $T_z=\pi(z_1,z_2,0,0)$, we have the equivalence between covariance and
	\[
		\cA(z_1,z_1,z_2,-z_2)=(z_1,z_2,0,0),
	\]
	which gives
	\begin{equation}\label{covinter}
		\cA=\begin{pmatrix}
			A_{11} & I-A_{11} & A_{13} & A_{13}\\
			A_{21} & -A_{21} & I- A_{11}^T & A_{11}^T\\
			A_{31} & -A_{31} & A_{33} & A_{33}\\
			A_{41} & -A_{41} & A_{43} & A_{43}
		\end{pmatrix}.
	\end{equation}
	The assertion follows imposing conditions (\ref{carSyymp}) to (\ref{covinter}).
\end{proof}

\begin{theorem}
	Let $\cA\in Sp(2d,\bR)$ have block decomposition (\ref{blockDA}). $W_\cA$ belongs to the Cohen's class if and only if $W_\cA$ is covariant, in the which case
	\begin{equation}\label{covCohen}
		W_\cA(f,g)=W(f,g)\ast \Sigma_\cA,
	\end{equation}
	with $\Sigma_\cA=\cF^{-1}(\Phi_{-B_\cA})$ and
	\begin{equation}\label{defBA}
		B_\cA:=\begin{pmatrix}
			A_{13} & \frac{1}{2}I-A_{11}\\
			\frac{1}{2}I-A_{11}^T & -A_{21}
		\end{pmatrix}.
	\end{equation}
\end{theorem}
\begin{proof}
	Since $W$ is covariant, if $W_\cA$ belongs to the Cohen's class, then
	\begin{align*}
		W_\cA(\pi(z)f,\pi(z)g)(x,\xi)&=W(\pi(z)f,\pi(z)g)\ast\Sigma_\cA(x,\xi)=(T_zW(f,g))\ast\Sigma_\cA(x,\xi)\\
		&=T_z(W(f,g)\ast\Sigma_\cA)(x,\xi).
	\end{align*}
	This proves also that, in general, Cohen's distributions are covariant, so it remains to prove that every covariant metaplectic Wigner distribution can be written as in (\ref{covCohen}). For, let $\cA\in Sp(2d,\bR)$ define a covariant $W_\cA$. An easy computation shows that 
	\begin{equation}\label{Cov1dec}
	\cA=V_{B_\cA}^TA_{1/2}. 
	\end{equation}
	Formula \ref{covCohen} follows using Example (\ref{es22}) (3), (\ref{Atau}) with $\tau=1/2$ and Theorem \ref{charcovidddi}. 
\end{proof}

\section{Characterization of Modulation Spaces via Metaplectic Wigner distributions}\label{sec:MS}
	Modulation spaces are usually introduced in terms of the STFT, as in paragraph \ref{subsec:MS}. However, as it is proved in \cite{MDG2011}, the classical Wigner distribution can be used for the purpose as well, i.e.
	\[
		f\in M^{p,q}(\rd) \qquad \Longleftrightarrow \qquad W(f,g)\in L^{p,q}(\rdd)
	\] 
	for some (hence, all) $g\in \cS(\rd)\setminus\{0\}$. In \cite{CR2022}, the authors prove that a particular classes of covariant metaplectic Wigner distributions, namely the covariant \textit{shift-invertible} Wigner distributions, play the same role in the more general case of weighted modulation spaces and suggest that shift-invertibility is the key property for a metaplectic Wigner distribution to characterize modulation spaces, but the proof of this conjecture is known only for modulation spaces $M^p_{v_s}$, $p\in[1,2]$ and $s\geq0$. 
	Here, we follow the pattern in \cite{CGR2022}, where the proof is extended to the \textit{metaplectic Wigner distributions of the classic type}. 
	
	\begin{remark}\label{rmark}
		The $\tau$-Wigner distribution $W_\tau$ of Examples \ref{exDist} (2) is shift-invertible if and only if $\tau\in(0,1)$. Actually, the Rihacek $W_0$ and the conjugate Rihacek distribution $W_1$ do not characterize modulation spaces, but 
		\[
		W_0(f,g)\in L^{p,q}_{v_s}(\rdd) \qquad \Longleftrightarrow \qquad f\in L^p_{v_s}(\rd) 
		\]
		and 
		\[
		W_1(f,g)\in L^{p,q}_{v_s}(\rdd) \qquad \Longleftrightarrow \qquad f\in \cF L^p_{v_s}(\rd).
		\]
		This justifies the intuition that shift-invertibility shall be fundamental for a metaplectic Wigner disribution to define modulation spaces.
		
		More generally, for given symplectic matrix $\cA\in Sp(2d,\bR)$, window $g\in \cS(\rd)\setminus\{0\}$, $p,q\in(0,+\infty]$ and $s\in\bR$, one can define
		\[
			M_{\cA,v_s}^{p,q}(\rd):=\Big\{f\in \cS'(\rd) : \norm{W_\cA(f,g)}_{L^{p,q}_{v_s}}<\infty \Big\}.
		\]
		Then, the question we address in this Section becomes whether $M_{\cA,v_s}^{p,q}(\rd)=M^{p,q}_{v_s}(\rd)$ for $\cA\in Sp(2d,\bR)$.
	\end{remark}
	
	\begin{definition}\label{defSI}
		Let $\cA\in Sp(2d,\bR)$. We say that $W_\cA$ is \textbf{shift-invertible} if
		\begin{equation}\label{defSIf}
			|W_\cA(\pi(w)f,g)|=|T_{E_\cA(w)}W_\cA(f,g)|, \qquad f,g\in L^2(\rd), \quad w\in \rdd
		\end{equation}
		for some matrix $E_\cA\in GL(2d,\bR)$, with
		\[
			T_{E_\cA(w)}W_\cA(f,g)(z):=W_\cA(f,g)(z-E_\cA(w)), \qquad z,w\in\rdd.
		\]
	\end{definition}
	
	We stress that Definition \ref{defSI} reflects the property of modulation spaces of being invariant with respect to time-frequency shifts. Moreover, if $W_\cA$ is shift-invertible with $\cA$ having decomposition (\ref{blockDA}), the direct computation of (\ref{defSIf}) shows that the matrix $E_\cA$ must be
	\begin{equation}\label{defEcA}
		E_\cA=\begin{pmatrix}
			A_{11} & A_{13}\\
			A_{21} & A_{23}
		\end{pmatrix}.
	\end{equation}
	In particular, if $\cA$ is also covariant, it must be
	\begin{equation}\label{EAcov}
		E_\cA=\begin{pmatrix}
			A_{11} & A_{13}\\
			A_{21} & I-A_{11}^T
		\end{pmatrix}.
	\end{equation}

	
	In \cite[Theorem 2.22]{CR2022} the authors show the following theorem for shift-invertible Wigner distributions as a first step towards the conjecture that shift-invertible Wigner distributions define modulation spaces. We point out that this result is partial in two senses: first, it only concerns $M^p_{v_s}$ spaces, and secondly the characterization is complete only for $1\leq p\leq2$. 
	
	\begin{theorem}\label{thrm}
		
		Let $g\in\cS(\rd)\setminus\{0\}$ and $\cA\in Sp(2d,\bR)$. \\
		(i) For $0<p\leq2$, if $f\in M^p_{v_s}(\rd)$, then $W_\cA(f,g)\in L^p_{v_s}(\rdd)$.\\
		(ii) If $W_\cA$ is also shift-invertible, then,\\
		(iia) For $s\geq0$, $1\leq p\leq2$, 
		\[
			f\in M^p_{v_s}(\rd) \qquad \Longleftrightarrow \qquad W_\cA(f,g)\in L^p_{v_s}(\rdd),
		\]
		with equivalence of norms $\norm{f}_{M^p_{v_s}}\asymp\norm{W_\cA(f,g)}_{L^p_{v_s}}$.\\	
		(iib) For $1\leq p\leq\infty$, if $W_\cA(f,g)\in L^p_{v_s}(\rdd)$, then $f\in M^p_{v_s}(\rd)$.\\
		(iic) For $0<p<1$, if $W_\cA(f,g)\in L^p_{v_s}(\rdd)$ and there exists a Gabor frame $\cG(\gamma,\Lambda)$ for $L^2(\rd)$ with $\gamma\in\cS(\rd)$ such that the sequence $W_\cA(f,\gamma)(\lambda)\in\ell^p_{v_s}(\Lambda)$, then $f\in M^p_{v_s}(\rd)$.
	\end{theorem}
	\begin{proof} $(i)$ follows directly by Theorem \ref{thm37}, using the continuous inclusions $M^p_{v_s}(\rd)\hookrightarrow M^p(\rd)$ ($s\geq0$), $M^p_{v_s}(\rdd)\hookrightarrow M^p_{v_s\otimes 1}(\rdd)\hookrightarrow L^p_{v_s}(\rdd)$, the last inclusion holding for $0<p\leq2$. \\
	$(ii)$ Let $W_\cA$ be shift-invertible. We first prove $(iib)$, $(iia)$ follows directly by $(i)$ and $(iib)$. For, assume that $W_\cA(f,g)\in L^p_{v_s}(\rdd)$. Applying Proposition \ref{prop36} with $g_3=g_1$, for all $w\in \bR^{2d}$,
	\begin{align*}
		|V_{g_1}f(w)|&\lesssim \frac{1}{|\langle g_2,g_1 \rangle|}|\langle W_\cA(f,g_1),W_\cA(\pi(w)g_1,g_2) \rangle_{L^2(\rdd)}|\\
		&\lesssim \int_{\rdd}|W_\cA(f,g_1)(u)||W_\cA(\pi(w)g_1,g_2)(u)|du\\
		&\lesssim \int_{\rdd}\int_{\rdd}|W_\cA(f,g_1)(u)||W_\cA(g_1,g_2)(u-E_\cA w)|du\\
		&\lesssim \int_{\rdd}|W_\cA(f,g_1)(u)||\cI W_\cA(g_1,g_2)(E_\cA w-u)|du.
	\end{align*}
	Therefore, Young's inequalities for $1\leq p\leq\infty$ and the invertibility of $E_\cA$ give
	\begin{align*}
		\norm{f}_{M^p_{v_s}}&\asymp\norm{V_{g_1}f}_{L^p_{v_s}}\lesssim \norm{W_\cA(f,g_1)\ast(\cI W_\cA(g_1,g_2)\circ E_\cA)}_{L^p_{v_s}}\\
		&\leq\norm{W_\cA(f,g_1)}_{L^p_{v_s}}\norm{W_\cA(g_1,g_2)}_{L^1_{v_s}},
	\end{align*}
	where $\norm{W_\cA(g_1,g_2)}_{L^1_{v_s}}<\infty$ since $W_\cA(g_1,g_2)\in \cS(\rdd)\subseteq L^1_{v_s}(\rdd)$.\\
	$(iic)$ Using $\norm{f}_{M^p_{v_s}}\asymp\norm{V_\gamma f}_{\ell^p_{v_s}}$ and arguing analogously to $(iib)$, with $g_3=\gamma$ instead of $g_3=g_1$, we get
	\[
		\norm{f}_{M^p_{v_s}}\lesssim\norm{W_\cA(f,\gamma)\ast(\cI W_\cA(\gamma,g_2)\circ E_\cA)}_{\ell^p_{v_s}}.
	\]
	Then, using the discrete Young's convolution inequality ($\ell^p_{v_s}\ast\ell^p_{v_s}\hookrightarrow\ell^p_{v_s}$, $0<p\leq1$, $s\geq0$), we get
	\[
		\norm{f}_{M^p_{v_s}}\lesssim\norm{W_\cA(f,\gamma)}_{\ell^p_{v_s}}\norm{W_\cA(\gamma,g_2)}_{\ell^p_{v_s}}
	\]
	and the assertion follows, since $W_\cA(\gamma,g_2)\in \cS(\rdd)$ implies $\norm{W_\cA(\gamma,g_2)}_{\ell^p_{v_s}}<\infty$.
	\end{proof}
	
	A complete characterization of $M^{p,q}_{v_s}$ spaces in terms of metaplectic Wigner distributions can be given in terms of the so-called \textit{Wigner-decomposable matrices}, cf. \cite{CGR2022}. Namely, if we require $\cA$ to be Wigner-decomposable, then shift-invertibility allows to characterize all $M^{p,q}_{v_s}$ spaces.
	
	\begin{definition}
		A matrix $\cA\in Sp(2d,\bR)$ is \textbf{Wigner-decomposable} if $\cA=V_C\cA_{FT2}\cD_L$ for some $C\in \bR^{2d\times2d}$ symmetric and $L\in GL(2d,\bR)$. The corresponding metaplectic Wigner distribution $W_\cA$ is \textbf{of the classic type}.
	\end{definition}
	
	\begin{remark}
		The term \textit{of the classic type} refers to the fact that these metaplectic Wigner distributions are immediate generalization of the classical Wigner distribution, which is of the classic type with $C=0$ and 
		\[
			L=\begin{pmatrix}
				I & \frac{1}{2}I\\
				I & -\frac{1}{2}I
			\end{pmatrix}.
		\]
		In \cite{CGR2022}, the authors define totally Wigner-decomposable matrices as Wigner-decomposable matrices with $C=0$. In \cite{CT2019, Z}, the same distributions are called \textbf{matrix Wigner distributions}.
	\end{remark}
	
	We turn to modulation spaces. We first need a definition.
	\begin{definition} A matrix $L\in \bR^{2d\times 2d}$ with block decomposition
	\begin{equation}\label{defL}
		L=\begin{pmatrix}
			L_{11} & L_{12}\\
			L_{21} & L_{22}
		\end{pmatrix},
	\end{equation}
	 is \textbf{right-regular} if $L_{12},L_{22}\in GL(d,\bR)$. \end{definition}
	 We compare right-regularity condition for matrices with shift-invertibility for metaplectic Wigner distributions.
	 
	 \begin{remark}
	 	If $L$ is right-regular, $L^{-T}$ is left-regular. Indeed, writing $L$ as in (\ref{defL}), by \cite[Theorem 2.1]{LS} with $L_{22}$ invertible,
		\[
			L^{-T}=\begin{pmatrix}
				(L_{11}-L_{12}L_{22}^{-1}L_{21})^{-T} & -(L_{11}-L_{12}L_{22}^{-1}L_{21})^{-T}L_{21}^TL_{22}^{-T}\\
				-L_{22}^{-T}L_{12}^T(L_{11}-L_{12}L_{22}^{-1}L_{21})^{-T} & L_{22}^{-T}+L_{12}^T(L_{11}-L_{12}L_{22}^{-1}L_{21})^{-T}L_{21}^TL_{22}^{-T}
			\end{pmatrix}
		\]
		with $\det(L)=\det(L_{22})\det(L_{11}-L_{12}L_{22}^{-1}L_{21})$. The invertibility of $L$ together with this determinant identity give the invertibility of $L_{11}-L_{12}L_{22}^{-1}L_{21}$, hence of the upper left-block. Since also $L_{12}$ is invertible, we get the invertibility of the lower-left block of $L^{-1}$ as well. We point out that the converse is also true, cf. \cite{Bayer}.
	 \end{remark}
	 
	 \begin{remark}\label{remprop43}
	 	Let $L\in GL(2d,\bR)$. Writing $L$ as in (\ref{defL}), computing explicitly $\cA_{FT2}\cD_L$, using conditions (\ref{carSyymp}) and comparing the blocks, it is possible to verify that the matrix $\cA=\cA_{FT2}\cD_L$ has block decomposition 
		\begin{equation}
			\cA=\begin{pmatrix}
				A_{11} & A_{12} & 0 & 0\\
				0 & 0 & A_{23} & A_{24}\\
				0 & 0 & A_{33} & A_{34}\\
				A_{41} & A_{42} & 0 & 0
			\end{pmatrix}
		\end{equation}
		with
		\begin{equation}\label{DeffL}
			L=\begin{pmatrix}
				A_{33}^T & A_{23}^T\\
				A_{34}^T & A_{24}^T
			\end{pmatrix}
		\end{equation}
		and that
		\begin{equation}\label{Lm1}
			L^{-1}=\begin{pmatrix}
				A_{11} & A_{12}\\
				-A_{41} & -A_{42}
			\end{pmatrix}.
		\end{equation}
		We refer to \cite[Proposition 4.3]{CGR2022} for the details.
	 \end{remark}
	
	\begin{lemma}\label{lemmaG}
		Let $\cA\in Sp(2d,\bR)$ be a Wigner-decomposable matrix with $\cA=V_C\cA_{FT2}\cD_L$ and block decomposition (\ref{blockDA}). The following statements are equivalent.\\
		(i) $L$ is right regular,\\
		(ii) $W_\cA$ is shift-invertible.\\
		In this case,
		\begin{equation}\label{espWAWD}
			W_\cA(f,g)(x,\xi)=\sqrt{|\det(L)|}|\det(A_{23})|^{-1}\Phi_C(x,\xi)e^{2\pi iA_{23}^{-1}\xi\cdot A_{33}^Tx}V_{\tilde g}f(c(x),d(\xi)),
		\end{equation}
		where
		\[
			\tilde g(t)=g(A_{24}^TA_{23}^{-T}t), \quad c(x)=(A_{33}-A_{34}A_{24}^{-1}A_{23})^Tx \quad and \quad d(\xi)=A_{23}^{-1}\xi.
		\]
	\end{lemma}
	\begin{proof}
	Observe that $W_{V_C\cA_{FT2}\cD_L}$ is shift-invertible if and only if $W_{\cA_{FT2}\cD_L}$ is shift-invertible, since the modulus in the definition of shift-invertibility drops the chirp function given by $\mu(V_C)$. Therefore, to prove the equivalence between $(i)$ and $(ii)$ we assume $L=0$ without loss of generality.
	
	The direct computation in \cite{CR2021} shows that for all $w=(w_1,w_2)\in\rdd$ and all matrices $\cA\in Sp(2d,\bR)$, 
	\[
		|W_\cA(\pi(w)f,g)|=|T_{(A_{11}w_1+A_{13}w_2,A_{21}w_1+A_{23}w_2)}W_\cA(f,g)|=|T_{E_\cA(w)}W_\cA(f,g)|.
	\]
	Hence, $W_\cA$ is shift-invertible if and only if $E_\cA$ defined as in (\ref{defEcA}) is invertible. Also, Remark \ref{remprop43} tells that $L$ is right-regular if and only if $A_{23}$ and $A_{24}$ are invertible. We have all the ingredients to prove the equivalence.\\
	$(i)\Longrightarrow(ii)$ Assume that $A_{23}$ and $A_{24}$ are invertible. $W_\cA$ is shift-invertible if and only if $E_\cA$ is invertible. The matrix $E_\cA$ for totally-Wigner decomposable matrices is 
	\[
		E_\cA=\begin{pmatrix}
			A_{11} & 0\\
			0 & A_{23}
		\end{pmatrix}
	\]
	and $A_{23}$ is invertible, so it remains to show that $A_{11}$ is also invertible. Since $L$ is right-regular, $L^{-T}$ is left-regular and the assertion follows by the fact that $A_{11}^T$ is one of the left blocks of $L^{-T}$, by (\ref{Lm1}).\\
	$(ii)\Longrightarrow$ To prove the converse, assume that $W_\cA$ is shift-invertible or, equivalently, that $E_\cA$ is invertible. Then, $A_{11}$ and $A_{23}$ are invertible. By the identity $A_{11}A_{23}^T=-A_{12}A_{24}^T$ (that follows by $L^{-1}L=I$ using Remark \ref{remprop43}), it follows that $A_{12}$ and $A_{24}$ are invertible. Hence, $A_{23}$ and $A_{24}$ are both invertible.
	
	Finally, formula (\ref{espWAWD}) follows writing explicitly the operators $\mu(V_C),\cF_2$ and $\mu(\cD_L)$.

	\end{proof}
	
		 Not only right-regularity characterizes shift-invertibility, but it also characterizes matrix Wigner distributions that are continuous functions. For the following result we refer to \cite[Corollary 1.2.6 and Theorem 1.2.9]{Bayer}
	 
	 \begin{theorem}
	 	If $L$ is right-invertible and $\cA=\A_{FT2}\cD_L$ is a matrix Wigner distribution, then $W_\cA(f,g)$ is uniformly continuous for all $f,g\in L^2(\rd)$. Moreover, if $\det(L_{22})\neq0$, but $\det(L_{12})=0$, then there exist $f,g\in L^2(\rd)$ such that $W_\cA(f,g)$ is not continuous. 
	 \end{theorem}
	 \begin{proof}
	 	The first part follows by (\ref{espWAWD}) and the uniform continuity of the STFT. The second part is constructive and can be found in \cite[Theorem 1.2.9]{Bayer}.
	 \end{proof}
	
	The following theorem proves that the shift-invertibility property characterizes all modulation spaces at least in the case of Wigner-decomposable matrices.
	
	\begin{theorem}
		Let $\cA\in Sp(2d,\bR)$ be a Wigner-decomposable matrix with $\cA=V_C\cA_{FT2}\cD_L$, with $L$ right-regular. Then, for all $g\in\cS(\rd)\setminus\{0\}$ and $0<p,q\leq\infty$, 
		\[
			f\in M^{p,q}_{v_s}(\rd)\qquad \Longleftrightarrow\qquad W_\cA(f,g)\in L^{p,q}_{v_s}(\rdd).
		\]
		For $1\leq p,q\leq\infty$, the window $g$ can be chosen in the wider class $M^1_{v_s}(\rd)\setminus\{0\}$.
	\end{theorem}
	\begin{proof}
	It is a direct consequence of (\ref{espWAWD}). We refer to \cite[Corollary 4.14]{CGR2022} for a more detailed proof.
	\end{proof}
	
\section{Lieb's Uncertainty Principles for classical Wigner distributions} \label{sec:UP}
	Wigner distributions of the classic type $W_\cA(f,g)$, having $\cA=V_C\cA_{FT2}\cD_L$, $C\in\bR^{2d\times2d}$ symmetric and $L\in GL(2d,\bR)$ right-regular, are basically STFT, cf. (\ref{espWAWD}). For this reason, uncertainty principles can be easily carried from STFT to matrix Wigner distributions. For the sake of simplicity, we limit ourselves to report on the trivial cases that extend the already known uncertainty principles for STFT, but it is mandatory to mention that many authors are studying uncertainty principles in the context of time-frequency distributions in the metaplectic framework, cf. \cite{Z2, ZH}.
	
	We first state a generalization of the weak uncertainty principle for matrix Wigner distributions. For the rest of this section, $\cA$ is a Wigner-decomposable matrix, with blocks as in (\ref{blockDA}), such that $L$ (which is given by its expression in (\ref{DeffL})) is right-regular.
	
	\begin{theorem}\label{thmLieb1}
		If $\norm{f}_2=\norm{g}_2=1$, and $U\subseteq\rdd$ and $\varepsilon>0$ are such that
		\[
			\int_U|W_\cA(f,g)(x,\xi)|^2dxd\xi\geq1-\varepsilon,
		\]
		then $|U|\geq 1-\varepsilon$.
	\end{theorem}
	\begin{proof}
		Let $f,g\in L^2(\rd)$, $\varepsilon$ and $U$ as in the statement. Recall that if $L$ is right regular, then $\det(L)=\det(A_{24})\det(A_{33}-A_{34}A_{24}^{-1}A_{23})$. In particular, the matrix associated to the linear transformation $(x,\xi)\in\rdd\mapsto(c(x),d(\xi))$ defined as in Lemma \ref{lemmaG} is invertible, its determinant being $\frac{\det(A_{33}-A_{34}A_{24}^{-1}A_{23})}{\det(A_{23})}$. Then, after a change of variable, 
		\begin{align*}
			1-\varepsilon&\leq \int_U|W_\cA(f,g)(x,\xi)|^2dxd\xi=\int_U\frac{|\det(L)|}{|\det(A_{23})|^{2}}|V_{\tilde g}f(c(x),d(\xi))|^2dxd\xi\\
			&=\left|\frac{\det(A_{24})}{\det(A_{23})}\right|\int_U|V_{\tilde g}f(y,\eta)|^2dyd\eta\leq\left|\frac{\det(A_{24})}{\det(A_{23})}\right| \norm{V_{\tilde g}f}_{\infty}^2|U|.
		\end{align*}
		Since $\norm{\tilde g}_2=\left|\frac{\det(A_{23})}{\det(A_{24})}\right|^{1/2}\norm{g}_2$, $|V_{\tilde g}f(y,\eta)|\leq \norm{\tilde g}_2\norm{f}_2=\left|\frac{\det(A_{23})}{\det(A_{24})}\right|^{1/2}$ and we get the assertion.
	\end{proof}
	
	
	\begin{theorem}\label{Lieb2}
		If $f,g\in L^2(\rd)$ and $2\leq p<\infty$, then
		\[
			\norm{W_\cA(f,g)}_p\leq \frac{|\det(A_{33}-A_{34}A_{24}^{-1}A_{23})|^{\frac{1}{2}-\frac{1}{p}}}{|\det(A_{23})|^{1/2}}\left(\frac{2}{p}\right)^{d/p}\norm{f}_2\norm{g}_2.
		\]
	\end{theorem}
	\begin{proof}
		The same steps of Theorem \ref{thmLieb1} and \cite[Theorem 3.3.2]{GrochenigBook} lead to the following computation:
		\begin{align*}
			\norm{W_\cA(f,g)}_p^p&=\frac{|\det(A_{24})|^{p/2}|\det(A_{33}-A_{34}A_{24}^{-1}A_{23})|^{\frac{p}{2}-1}}{|\det(A_{23})|^p}\norm{V_{\tilde g}f}_{p}^p\\
			&\leq\frac{|\det(A_{24})|^{p/2}|\det(A_{33}-A_{34}A_{24}^{-1}A_{23})|^{\frac{p}{2}-1}}{|\det(A_{23})|^p}\left(\frac{2}{p}\right)^{d}\norm{f}_2^p\norm{\tilde g}_2^p.
		\end{align*}
		Using again $\norm{\tilde g}_2=\left|\frac{\det(A_{23})}{\det(A_{24})}\right|^{1/2}	\norm{g}_2$,
		\[
			\norm{W_\cA(f,g)}_p^p\leq\frac{|\det(A_{33}-A_{34}A_{24}^{-1}A_{23})|^{\frac{p}{2}-1}}{|\det(A_{23})|^{p/2}}\left(\frac{2}{p}\right)^{d}\norm{f}_2^p\norm{ g}_2^p,
		\]
		and we are done.
	\end{proof}
	
	Analogously, the opposite inequality for $1\leq p\leq2$ can be proved.
		\begin{theorem}\label{thmF2}
		If $f,g\in L^2(\rd)$ and $1\leq p\leq 2$, then
		\[
			\norm{W_\cA(f,g)}_p\geq \frac{|\det(A_{33}-A_{34}A_{24}^{-1}A_{23})|^{\frac{1}{2}-\frac{1}{p}}}{|\det(A_{23})|^{1/2}}\left(\frac{2}{p}\right)^{d/p}\norm{f}_2\norm{g}_2.
		\]
	\end{theorem}
	
	Observe that Theorem \ref{Lieb2} and Theorem \ref{thmF2} for $p=2$ give (\ref{moyal}). As in the case of STFT, Theorem \ref{Lieb2} can be used to improve Theorem \ref{thmLieb1}.
	
	\begin{theorem}
		Under the same assumptions of Theorem \ref{thmLieb1}, for all $p>2$
		\[
			|U|\geq(1-\varepsilon)^{\frac{p}{p-2}}\frac{|\det(A_{23})|^{\frac{p}{p-2}}}{|\det(A_{33}-A_{34}A_{24}^{-1}A_{23})|}\left(\frac{p}{2}\right)^{\frac{2d}{p-2}}.
		\]
	\end{theorem}
	\begin{proof}
		We follow the same steps as in \cite[Theorem 3.3.3]{GrochenigBook}. Using H\"older's inequality with exponent $q=p/2$, we have:
		\begin{align*}
			1-\varepsilon&\leq \int_U|W_\cA(f,g)(x,\xi)|^2dxd\xi\\
			&\leq \left(\int_{\rdd}|W_\cA(f,g)(x,\xi)|^pdxd\xi\right)^{2/p}\left(\int_{\rdd}\chi_U(x,\xi)dxd\xi\right)^{\frac{p-2}{p}}\\
			&\leq \frac{|\det(A_{33}-A_{34}A_{24}^{-1}A_{23})|^{1-\frac{2}{p}}}{|\det(A_{23})|}\left(\frac{2}{p}\right)^{2d/p}|U|^{\frac{p-2}{p}},
		\end{align*}
		that is the assertion.
	\end{proof}
	
	We conclude this section with a Donoho-Stark-like uncertainty principle that was observed in \cite[Theorem 1.4.3]{Bayer}. 
	\begin{theorem}
		Let $\cA$ be a Wigner-decomposable matrix with $L$ right-regular. Let $f,g\in L^2(\rd)$. If the (Lebesgue) measure of $\supp(W_\cA(f,g))$ is finite, then either $f=0$ or $g=0$.
	\end{theorem}

\section{Metaplectic Quantizations for Pseudodifferential Operators}\label{sec:MPO}
Once a time-frequency representation is defined, one can introduce a new quantization law for pseudodifferential operators. Recall that a quantization law is a mapping $\sigma\in\cS'(\rdd)\mapsto T_\sigma$, where $T_\sigma:\cS(\rd)\to\cS'(\rd)$ depends on the choice of the time-frequency representation. In the following definition, which was firstly given in \cite{CR2022}, the classical Weyl quantization is generalized up to include metaplectic Wigner distributions.

\begin{definition}
	Let $\cA\in Sp(2d,\bR)$. The \textbf{metaplectic Wigner distribution with symbol} $a\in\cS'(\rdd)$ is the operator $Op_\cA(a):\cS(\rd)\to\cS'(\rd)$ defined for all $f\in\cS(\rd)$ by duality as
	\[
		\langle Op_\cA(a)f,g\rangle:=\langle a,W_\cA(g,f)\rangle,\qquad g\in \cS(\rd).
	\]
\end{definition}
If $\cA=A_{1/2}$, we recover the well-known Weyl quantization:
\[
Op_w(a)f(x)=\int_{\rdd}e^{2\pi i(x-y)\xi}a\Big(\frac{x+y}{2},\xi\Big)f(y)dyd\xi.
\]

In this section, we report on the main properties of metaplectic pseudodifferential operators. First, we show a change of quantization formula, which leads to a kernel formula, useful to study $L^p$ boundedness of metaplectic pseudodifferential operators.

\begin{proposition}\label{propCambioQuant}
	Let $\cA,\cB\in Sp(2d,\bR)$ and $a,b\in\cS'(\rdd)$. Then,
	\begin{equation}\label{condKer}
		Op_\cA(a)=Op_\cB(b) \qquad \Longleftrightarrow \qquad b=\mu(\cB)\mu(\cA)^{-1}(a).
	\end{equation}
	As a consequence, for all $f\in \cS(\rd)$
	\[
		\langle Op_\cA(a)f,g\rangle = \langle k_\cA(a),g\otimes\bar f\rangle \qquad g\in\cS(\rd)
	\]
	with $k_\cA(a)=\mu(\cA)^{-1}(a)$.
\end{proposition}
\begin{proof}
	For $f,g\in\cS(\rd)$,
	\[
		\langle Op_\cA(a)f,g \rangle=\langle a,\mu(\cA)(g\otimes\bar f)\rangle=\langle\mu(\cA)^{-1}a,g\otimes\bar f\rangle
	\]
	and on the other hand,
	\[
		\langle Op_\cB(b)f,g \rangle=\langle\mu(\cB)^{-1}b,g\otimes\bar f\rangle.
	\]
	This identity extends trivially to $L(\cS)$, which is dense in $\cS(\rdd)$. Hence,
	\[
		\langle\mu(\cA)^{-1}a,F\rangle=\langle\mu(\cB)^{-1}b,F\rangle
	\]
	for all $F\in\cS(\rdd)$. Hence, $Op_\cA(a)=Op_\cB(b)$ if and only if $\mu(\cA)^{-1}a=\mu(\cB)^{-1}b$. For the kernel formula, choose $\cB=I$ in (\ref{condKer}).
\end{proof}

The classical Wigner distribution satisfies the intertwining formula
\[
	W(Op_w(a)f,g)=Op_w(b)W(f,g) \qquad (a\in S^0_{0,0}(\rdd))
\]
and $b(x,\xi,u,v):=a(x-v/2,\xi+u/2)$. The following result shows that this relation can be extended to metaplectic Wigner distributions in a very general fashion, cf. \cite[Corollary 3.4]{CGR2022}. All the identities in Proposition \ref{corollarioallacomm} have to be intended up to a unitary phase factor.

\begin{proposition}\label{corollarioallacomm}
		Consider $\cA\in Sp(4d,\mathbb{R})$, $\mathcal{B}\in Sp(2d,\mathbb{R})$ and $a\in\mathcal{S}'(\mathbb{R}^{2d})$. Then, for all $\mathcal{B}_0\in Sp(4d,\mathbb{R})$, $f,g\in\mathcal{S}(\rd)$,
		\begin{equation}\label{commOpW}
			W_\cA(Op_{\mathcal{B}}(a)f,{g})=Op_{\mathcal{B}_0}(\mu(\mathcal{B}_0{A}_{1/2}^{-1})((\mu({A}_{1/2}\mathcal{B}^{-1})a)\otimes 1)\circ\cA^{-1})W_\cA(f,g).
		\end{equation}
		In particular, \\
		(i) if $\mathcal{B}_0=A_{1/2}$, then
		\begin{equation}\label{commOpWParticolare}
				W_\cA(Op_\mathcal{B}(a)f,g)=Op_{w}(((\mu({A}_{1/2}\mathcal{B}^{-1})a)\otimes 1)\circ\cA^{-1})W_\cA(f,g);
		\end{equation}
		\noindent
		(ii) if $\mathcal{B}_0=A_{1/2}$ and $\mathcal{B}=A_{1/2}$, then
		\begin{equation}\label{commOpWParticolare2}
			W_\cA(Op_{w}(a)f,g)=Op_{w}((a\otimes 1)\circ\cA^{-1})W_\cA(f,g).
		\end{equation}
	\end{proposition}
	\begin{proof}
		Recall that, in general, $\mu(\cA)\mu(\cB)=\mu(\cA\cB)$ up to a unitary constant. By \cite[Lemma 4.1]{CR2021}, for all $a\in\cS'(\rdd)$ and all $f,g\in\cS(\rd)$,
		\[
			(Op_w(a)f)\otimes \bar g=Op_w(a\otimes 1)(f\otimes \bar g).
		\]
		Using $\mu(\cA)Op_w(\sigma)\mu(\cA)^{-1}=Op_w(\sigma\circ\cA^{-1})$, we get
		\begin{align*}
			W_\cA(Op_\cB(a)f,g)&=\mu(\cA)(Op_\cB(a)f\otimes\bar g)=\mu(\cA)(Op_w(\mu(A_{1/2}\cB^{-1})a)f\otimes\bar g)\\
			&=\mu(\cA)((Op_w(\mu(A_{1/2}\cB^{-1})a)\otimes 1)(f\otimes\bar g))\\
			&=Op_w((((\mu(A_{1/2}\cB^{-1}))a)\otimes 1)\circ\cA^{-1})\mu(\cA)(f\otimes\bar g)\\
			&=Op_w((((\mu(A_{1/2}\cB^{-1}))a)\otimes 1)\circ\cA^{-1})W_\cA(f,g).
		\end{align*}
		The assertion follows by Proposition \ref{propCambioQuant}.
	\end{proof}
	
	A boundedness result on $M^p_{v_s}$ spaces follows directly by Theorem \ref{thm37}.
	
	\begin{theorem}\label{thmF}
		Let $1\leq p<\infty$, $s\geq0$ $a\in M^{p'}_{v_{-s}}(\rdd)$ and $\cA\in Sp(2d,\bR)$. Then, $Op_\cA(a):\cS(\rd)\to\cS'(\rd)$ extends to a bounded operator from $M^p_{v_s}(\rd)$ to $M^{p'}_{v_{-s}}(\rd)$. In particular, if $a\in M^{p'}(\rdd)$, $Op_\cA(a)$ extends to a bounded operator from $M^p(\rd)$ to $M^{p'}(\rd)$.
	\end{theorem}
	\begin{proof}
		By Theorem \ref{thm37} and the continuous inclusion $M^p_{v_s}(\rd)\hookrightarrow M^p(\rd)$ holding for $s\geq0$, we get $\norm{W_\cA(f,g)}_{M^p_{v_s}}\lesssim\norm{f}_{M^p_{v_s}}\norm{g}_{M^p_{v_s}}$. By duality,
		\begin{align*}
			|\langle Op_\cA(a)f,g\rangle|&=|\langle a,W_\cA(g,f)\rangle|\lesssim \norm{a}_{M^{p'}_{v_{-s}}}\norm{W_\cA(g,f)}_{M^p_{v_s}}\\
			&\lesssim\norm{a}_{M^{p'}_{v_{-s}}}\norm{f}_{M^p_{v_{s}}}\norm{g}_{M^p_{v_s}},
		\end{align*}
		so that
		\[
			\norm{Op_\cA(a)f}_{M^{p'}_{v_{-s}}}\lesssim\norm{a}_{M^{p'}_{v_{-s}}}\norm{f}_{M^p_{v_s}}.
		\]
		The continuity on the unweighted modulation spaces follows by choosing $s=0$. 
	\end{proof}
	
	As a consequence, we obtain the following version of Calderon-Vaillancourt Theorem for metaplectic pseudodifferential operators:
	
	\begin{corollary}
	Let $a\in M^{p'}(\rdd)$ with $p\in[2,+\infty)$, then $Op_\cA(a)$ extends to a bounded operator on $L^2(\rd)$.
	\end{corollary}
	\begin{proof}
	If $p\geq2$ and $s\geq0$, then $M^{p'}(\rdd)\hookrightarrow M^2(\rdd)$ and it follows by Theorem \ref{thmF} that $Op_\cA(a)$ extends to a bounded operator from $M^2(\rd)=L^2(\rd)$ to itself.
	\end{proof}
	
	We conclude this Section with the boundedness properties of metaplectic pseudodifferential operators on modulation spaces in the case in which $W_\cA$ is in the form $\cF_2\mathfrak{T}_L$ for an appropriate choice of $L$. Recall that for $\cA=A_{1/2}$, the following proposition holds, cf. \cite{GrochenigBook, Toft1, Toft2}:
	\begin{proposition}\label{propWBC}
		If $0<p,q,r\leq\infty$ with $r=\min\{1,p,q\}$, $s\in\bR$ and $\sigma\in M^{\infty,r}_{1\otimes v_{|s|}}(\rdd)$. Then, $Op_w:\cS(\rd)\to\cS'(\rd)$ extends to a bounded operator on $\cM^{p,q}_{v_s}(\rd)$.
	\end{proposition}
	
	\begin{theorem}
		Assume that $\cB_\cA$ defined as in (\ref{defBA}) is invertible. For $0<p,q\leq\infty$, set $r=\min\{1,p,q\}$. If $a\in M^{\infty,r}_{1\otimes v_s}(\rdd)$, $s\geq0$, then $Op_\cA(a):\cS(\rd)\to\cS'(\rd)$ extends to a bounded operator on $\cM^{p,q}_{v_s}(\rd)$.
	\end{theorem}
	\begin{proof}
		Observe that $\cF^{-1}\Phi_{B_\cA}(z,\zeta)\in M^{1,\infty}(\rdd)$ by \cite[Lemma 5.2]{CGR2022}. 
		Also, $W_\cA$ is covariant by Theorem \ref{charcovidddi} with $W_\cA(f,g)=W(f,g)\ast\cF^{-1}(\Phi_{-B_{\cA}})$. Therefore,
		\begin{align*}
			\langle Op_\cA(a)f,g\rangle&=\langle a,W_\cA(g,f)\rangle=\langle a,W(g,f)\ast\cF^{-1}\Phi_{-B_\cA}\rangle\\
			&=\langle \hat a,\cF(W(g,f))\Phi_{-B_\cA}\rangle=\langle \hat a\Phi_{B_\cA},\cF(W(g,f))\rangle\\
			&=\langle a\ast \cF^{-1}\Phi_{B_\cA},W(g,f)\rangle=\langle Op_w(a\ast \cF^{-1}\Phi_{B_\cA})f,g\rangle.
		\end{align*}
		By the convolution inequalities for modulation spaces, cf. \cite[Proposition 3.1]{BCN20}, we have
		\[
			\norm{a\ast\cF^{-1}\Phi_{B_\cA}}_{M^{\infty,r}_{1\otimes v_s}}\lesssim\norm{a}_{M^{\infty,r}_{1\otimes v_s}}\norm{\cF^{-1}\Phi_{B_\cA}}_{M^{1,\infty}},
		\]
		and the conclusion follows by Proposition \ref{propWBC}.
	\end{proof}

\section*{Conclusions and Open Problems}
We illustrated how Metaplectic Wigner distributions are defined as natural generalizations of the classical Wigner distribution in the metaplectic framework. As aforementioned, they seem to be the key subject to fully understand the nature of time-frequency analysis, but many of the topics exposed in this survey are still poorly understood. Which spaces can be written as $M^{p,q}_{\cA,v_s}(\rd)$, cf. Remark \ref{rmark}? For instance, it is clear that Wiener spaces can be represented as this kind of spaces too. The characterization of modulation spaces in terms of shift-invertibility (namely, Theorem \ref{thrm} $(iic)$) left another open question: how can the definition of Gabor frames be restated for (shift-invariant) metaplectic Wigner distributions instead of the STFT? We have seen that Lieb's uncertainty principle admits different bounds, in the setting of Wigner-decomposability. Since decomposability is just a first generalization of the Wigner distribution, one may wonder how far it is possible to go to minimize the effect of uncertainty principles with the employment of more sophisticated metaplectic time-frequency representations. 

\section*{Acknowledgements}
The author is partially founded by HES-SO School of Engineering and he is member of the \textit{Gruppo Nazionale per l’AnalisiMatematica, la Probabilità e le loro Applicazioni} (GNAMPA) of the \textit{Istituto Nazionale di Alta Matematica} (INdAM) and \textit{The Sense Institute of Innovation}, (Lausane/Sion). The author would like to thank Prof. E. Cordero and Prof. L. Rodino for Their suggestions and contributions to this work.

\end{document}